\documentclass[12pt]{article}
\usepackage{}
\usepackage[numbers,sort&compress]{natbib}

\usepackage{color}
\usepackage{threeparttable}

\usepackage[english]{babel}
\usepackage{t1enc}
\usepackage{amsthm,amsmath,amssymb}
\usepackage{mathrsfs}
\usepackage[mathscr]{eucal}
\usepackage[latin1]{inputenc}
\usepackage[english]{babel}
\usepackage{latexsym}
\usepackage{graphicx}
\usepackage[noend]{algpseudocode}
\usepackage{algorithmicx,algorithm}
\usepackage{lineno}
\usepackage{booktabs}
\usepackage{multirow}
\newtheorem{theorem}{Theorem}
\newtheorem{corollary}{Corollary}
\newtheorem{proposition}{Proposition}
\newtheorem{lemma}{Lemma}
\newtheorem{remark}{Remark}
\newtheorem{observation}{Observation}

\newtheorem{definition}{Definition}

\def\sqrt{\textup{sqrt}}
\def \T{\textup{T}}

\def \diag{\textup{diag}}

\def \Res{\textup{Res}}
\makeatletter
\newcommand{\rmnum}[1]{\romannumeral #1}
\newcommand\restr[2]{{
		\left.\kern-\nulldelimiterspace 
		#1 
		\right|_{#2} 
}}
\newcommand{\Rmnum}[1]{\expandafter\@slowromancap\romannumeral #1@}
 
\makeatother
\title{On the determinant of the walk matrix of the rooted product with a path}
\author{\small Zhidan Yan\quad\quad Wei Wang\thanks{Corresponding author: wangwei.math@gmail.com}
\\
{\footnotesize School of Mathematics, Physics and Finance, Anhui Polytechnic University, Wuhu 241000, China}}
\date{}
\begin{document}
 \maketitle

\begin{abstract}
 For an $n$-vertex graph $G$, the walk matrix of $G$, denoted by $W(G)$, is the matrix $[e,A(G)e,\ldots,(A(G))^{n-1}e]$, where $A(G)$ is the adjacency matrix of $G$ and $e$ is the all-ones vector. For two integers $m$  and  $\ell$ with $1\le \ell\le (m+1)/2$, let $G\circ P_m^{(\ell)}$ be the rooted product of $G$ and the path  $P_m$ taking the $\ell$-th vertex of $P_m$ as the root, i.e., $G\circ P_m^{(\ell)}$ is a graph  obtained from $G$ and $n$ copies of the path $P_m$  by identifying the $i$-th vertex of $G$ with the $\ell$-th vertex (the root vertex) of the $i$-th copy of $P_m$ for each $i$. We prove that, 
\begin{equation*}
\det W(G\circ P_m^{(\ell)}) =
	\begin{cases}
			\pm (\det A(G))^{\lfloor\frac{m}{2}\rfloor}(\det W(G))^m & \text{if $\gcd(\ell,m+1)=1$,} \\
		0&\text{otherwise.}
	\end{cases}
\end{equation*}
This extends a recent result established in [Wang et al. Linear Multilinear Algebra 72 (2024): 828--840] which corresponds to the special case $\ell=1$. As a direct application, we prove that if $G$ satisfies $\det A(G)=\pm 1$ and $\det W(G)=\pm 2^{\lfloor n/2\rfloor}$,  then for any sequence of integer pairs ${(m_i,\ell_i)}$ with  $\gcd(\ell_i,m_i+1)=1$ for each $i$, all the graphs in the family
\begin{equation*}
	G\circ P_{m_1}^{(\ell_1)}, (G\circ P_{m_1}^{(\ell_1)})\circ P_{m_2}^{(\ell_2)}, ((G\circ P_{m_1}^{(\ell_1)})\circ P_{m_2}^{(\ell_2)})\circ P_{m_3}^{(\ell_3)},\ldots
\end{equation*}
are determined by their generalized spectrum.\\

\noindent\textbf{Keywords}: walk matrix; rooted product graph;  generalized spectrum;  Chebyshev polynomials.\\
\noindent
\textbf{AMS Classification}: 05C50
\end{abstract}
\section{Introduction}
\label{intro}
Let $G$ be a simple graph with vertex set $\{1,\ldots,n\}$.  The \emph{adjacency matrix} of $G$ is the $n\times n$ symmetric matrix $A=(a_{i,j})$, where $a_{i,j}=1$ if $i$ and $j$ are adjacent;  $a_{i,j}=0$ otherwise.  For a graph $G$, the \emph{walk matrix} of $G$ is
\begin{equation*}
W(G):=[e,Ae,\ldots,A^{n-1}e],
\end{equation*}
where $e$ is the all-ones vector. This particular kind of  matrix has many interesting properties and is related to several important problems such as the spectral characterization of graphs \cite{wang2017,wwz2023,ww2024}, the controllability \cite{godsil2012,rourke2016}  and reconstructibility \cite{gm1981} of graphs. For more studies on  the walk matrices of graphs, we refer to \cite{liu2022,choi2021,moon2023,wang2021}.

Let $H^{(v)}$ be a rooted graph with a root vertex $v$. The \emph{rooted product graph} \cite{godsil1978,schwenk1974} of $G$ and $H^{(v)}$, denoted by $G\circ H^{(v)}$, is a graph obtained from $G$ and $n$ copies of $H^{(v)}$ by identifying the  root vertex of the $i$-th copy of $H^{(v)}$ with  vertex $i$ of $G$ for $i=1,2,\ldots,n$. Let $P_m$ be the path of order $n$ whose vertices are labeled naturally as $1,2,\ldots,m$. We are mainly concerned with the family of graphs $G\circ P_m^{(\ell)}$ for general integers $m$ and $\ell$ with $1\le \ell\le m$. Owing to the symmetry of the path $P_m$, we may safely assume that $\ell\le(m+1)/2$.   For the special case $\ell=1$, the authors established the following formula concerning  the determinant of  $W(G\circ P_m^{(\ell)})$, which was conjectured in \cite{mao2022}.
\begin{theorem}[\cite{wym2024}]\label{Pm1}
	For any graph $G$ and integer $m\ge 2$,
\begin{equation*}
\det W(G\circ P_m^{(1)})=\pm (\det A(G))^{\lfloor\frac{m}{2}\rfloor}(\det W(G))^m.
\end{equation*} 
\end{theorem}
Theorem \ref{Pm1} has an important application in the generalized spectral characterization of graphs. For a graph $G$, the  spectrum of $G$ means the multiset of eigenvalues of the adjacency matrix $A(G)$. The spectrum of $G$ together with the spectrum of the complement $\overline{G}$ is referred to as the \emph{generalized spectrum} of $G$. A graph $G$ is \emph{determined by its generalized spectrum} (DGS for short) if any graph sharing the same generalized spectrum as $G$ must be isomorphic to $G$. It is known that for any $n$-vertex graph $G$, the determinant of $W(G)$ always has the form $2^{\lfloor n/2\rfloor}b$ for some  integer $b$ \cite{wang2013}. Furthermore,  a theorem of Wang \cite{wang2017} states that if $b$ (i.e., $2^{-\lfloor n/2\rfloor}\det W(G)$)  is odd and square-free, then $G$ is DGS. In particular, if $\det W(G)=\pm 2^{\lfloor n/2\rfloor}$ (i.e., $b=\pm 1$), then $G$ is DGS.
\begin{definition}
For an integer $n$, we define $\mathcal{F}_n$ to be the collection of all $n$-vertex graphs $G$ such that $\det A(G)=\pm 1$ and $\det W(G)=\pm 2^{\lfloor n/2\rfloor}$.
\end{definition}
We remark that $\mathcal{F}_n$ is empty if $n$ is odd. Indeed, according to Sachs' coefficient theorem, a graph $G$ with an odd number of vertices must satisfy $\det A(G)\equiv 0\pmod{2}$, contradicting the requirement that $\det A(G)=\pm 1$. Mao and Wang \cite{mao2022} showed that if $\det A(G)=\pm 1$ then $\det A(G\circ P_m^{(1)})=\pm 1$ for any $m\ge 2$. Let 
\begin{equation}
	\mathcal{F}=\bigcup_{n\text{~even}} \mathcal{F}_n
\end{equation}
Then it is easy to see from Theorem \ref{Pm1}  that $\mathcal{F}$ is \emph{closed} for the  operations $G\circ P_m^{(1)}$, that is, for all integers $m$,  it holds that $G\circ P_m^{(1)}\in \mathcal{F}$ whenever $G\in \mathcal{F}$.  Noting that each graph in $\mathcal{F}$ is DGS, we have obtained a method of constructing large DGS-graphs from smaller ones. We summarize this construction in the following theorem.

\begin{theorem}[\cite{wym2024}]\label{dgscon1}
	Let $G\in \mathcal{F}$. Then for any integer sequence ${m_i}$ with each $m_i\ge 2$, all the graphs in the family 
	\begin{equation*}
		G\circ P_{m_1}^{(1)}, (G\circ P_{m_1}^{(1)})\circ P_{m_2}^{(1)},  ((G\circ P_{m_1}^{(1)})\circ P_{m_2}^{(1)})\circ P_{m_3}^{(1)}, \cdots
	\end{equation*}
	are DGS.
\end{theorem}

The main aim of this paper is to generalize the above two theorems from $G\circ P_m^{(1)}$ to $G\circ P_m^{(\ell)}$ for general $\ell$. The main result of this paper is the following.
\begin{theorem}\label{main}
	Let $n,m\ge 2$. For any $n$-vertex graph $G$ and integer $\ell$ with $1\le \ell\le (m+1)/2$,
		\begin{equation*}
			\det W(G\circ P_m^{(\ell)})=\begin{cases}
				\pm (\det A(G))^{\lfloor\frac{m}{2}\rfloor}(\det W(G))^m&\text{if $\gcd(\ell,m+1)=1$,} \\
				0&\text{otherwise.}
			\end{cases}
		\end{equation*}
\end{theorem} 
Using Theorem \ref{main} as a new tool, the method of constructing DGS-graphs stated in Theorem \ref{dgscon1} can be naturally extended.
\begin{theorem}\label{cons2}
		Let $G\in \mathcal{F}$, then for any  sequence of integer pairs ${(m_i,\ell_i)}$ with each $1\le \ell_i\le (m_i+1)/2$ and $\gcd(\ell_i,m_i+1)=1$, all graphs in the family 
	\begin{equation*}
		G\circ P_{m_1}^{(\ell_1)}, (G\circ P_{m_1}^{(\ell_1)})\circ P_{m_2}^{(\ell_2)},  ((G\circ P_{m_1}^{(\ell_1)})\circ P_{m_2}^{(\ell_2)})\circ P_{m_3}^{(\ell_3)}, \cdots
	\end{equation*}
	are DGS.
	\end{theorem}

\section{Eigenvalues and eigenvectors of $G\circ P_m^{(\ell)}$}\label{comp_ev}

\begin{figure}
	\centering
	\includegraphics[height=4cm]{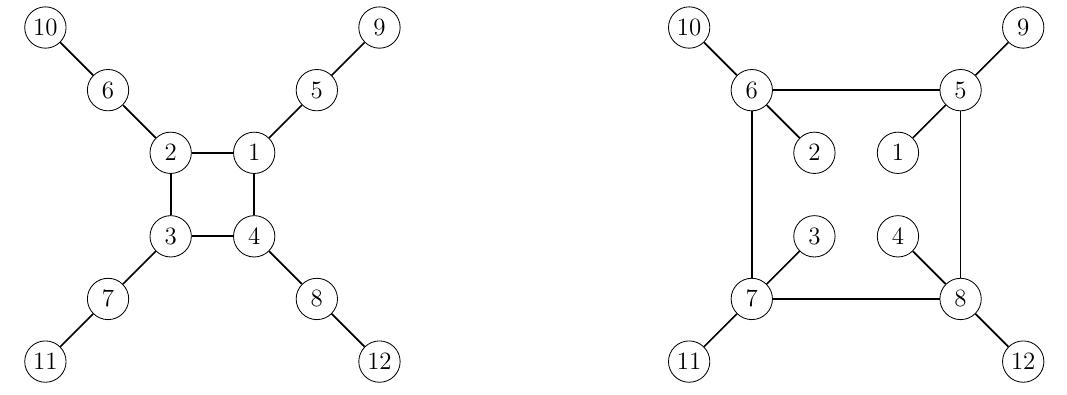}
	\caption{ $C_4\circ P_3^{(1)}$ (left) and $C_4\circ P_3^{(2)}$ (right).}
	\label{twographs}
\end{figure}
Let $U_m(x)$ be the $m$-th Chebyshev polynomial of the second kind, defined by
\begin{equation*}
	U_m(\cos \theta)=\frac{\sin(m+1)\theta}{\sin \theta}.
\end{equation*}
Let $S_m(x)=U_m(x/2)$. Then the sequence $\{S_m(x)\}$ satisfies the three-term recurrence relation:
 $S_{m}(x)=x S_{m-1}(x)-S_{m-2}(x)$ with initial values $S_0(x)=1$ and $S_1(x)=x$. We note that $S_m(x)$ is a monic polynomial with integral coefficients and is usually referred to as the \emph{renormalized} Chebyshev polynomial. For a graph $G$, let $\phi(G;x)=\det (xI-A(G))$ be the characteristic polynomial of $G$.  It is well known that $\phi(P_m;x)=S_m(x)$.

\begin{definition}\normalfont{
	Let $A=(a_{ij})$ be an $m\times n$ matrix and $B$  a $ p \times  q$ matrix. The \emph{Kronecker product} $A\otimes B$ is the  block matrix:
	$$A\otimes B=\begin{bmatrix}
	a_{11}B&\cdots&a_{1n}B\\
	\vdots&\ddots&\vdots\\
	a_{m1}B&\cdots&a_{mn}B
	\end{bmatrix}.
	$$
}
\end{definition}
By appropriately labeling the vertices in $G\circ P_m^{(\ell)}$ (see Figure \ref{twographs} for an illustration), the  adjacency matrix of $G\circ P_m$ can be described easily via the Kronecker product.
\begin{observation}\label{adjGP}
	$A(G\circ P_m^{(\ell)})=A(P_m)\otimes I_n+D_\ell\otimes A(G),$
	where $I_n$ is the identity matrix of order $n$ and $D_\ell$ is the diagonal matrix $\diag(\underbrace{0,\ldots,0}_{\ell-1},1,\underbrace{0,\ldots,0}_{m-\ell})$ of order $m$.
	\end{observation}
For a graph $H$ and a vertex $v\in V(H)$, we use $H-v$ to denote the graph obtained from $H$ by deleting the vertex $v$ (and all edges incident with $v$). We need the following formula for the characteristic polynomial of a rooted product graph.
\begin{lemma}[\cite{schwenk1974,godsil1978,gutman1980}]
We have $\phi(G\circ H^{(v)};x)=\prod_{i=1}^n(\phi(H;x)-\lambda_i \phi(H-v;x))$, where $\lambda_1,\ldots,\lambda_n$ are eigenvalues of $G$.
\end{lemma}
Noting that $\phi(P_m-\ell;x)=\phi(P_{\ell-1};x)\phi(P_{m-\ell};x)=S_{\ell-1}(x)S_{m-\ell}(x)$, the following corollary on $\phi(G\circ P_{m}^{(\ell)};x)$ is immediate.
\begin{corollary}\label{chagpk}
	We have $\phi(G\circ P_m^{(\ell)};x)=\prod_{i=1}^{n}(S_m(x)-\lambda_i S_{\ell-1}(x)S_{m-\ell}(x))$, where $\lambda_1,\ldots,\lambda_n$ are the eigenvalues of $G$.
\end{corollary}

\begin{proposition}\label{nsc}
	Let $n,m\ge 2$ and $1\le \ell\le (m+1)/2$. Suppose that $G$ is an $n$-vertex graph whose eigenvalues are all simple.  Then the graph $G\circ P_{m}^{(\ell)} $ has only simple eigenvalues if and only if $\gcd(\ell,m+1)=1$.
\end{proposition}
\begin{proof}
 The zeroes of  polynomials $S_m(x)$, $S_{\ell-1}(x)$ and $S_{m-\ell}(x)$ are $A:=\{a_p=2\cos\frac{p\pi}{m+1}\colon\,1\le p\le m \}$, and $B:=\{b_q=2\cos\frac{q\pi}{\ell}\colon\,1\le q\le \ell-1 \}$ and $C:=\{c_r=2\cos\frac{r\pi}{m+1-\ell}\colon\,1\le r\le m-\ell\}$, respectively. Let $d=\gcd(\ell,m+1)$. Suppose that $d>1$. Pick $p^*=(m+1)/d$ and $q^*=\ell/d$. Then clearly $a_{p^*}=b_{q^*}=2\cos\frac{\pi}{d}$, which, combining with the assumption $n\ge 2$,  implies that  $\phi(G\circ P_m^{(\ell)};x)$ has $2\cos \frac{\pi}{d}$ as a multiple root by Corollary \ref{chagpk}. This proves the necessity of  Proposition \ref{nsc}.
 
 Now we suppose that $d=1$. We prove the sufficiency by establishing the following two claims.
 
 \noindent\textbf{Claim 1}: The three sets $A$, $B$ and  $C$ are pairwise disjoint.
 
  Suppose to the contrary that $A\cap B$ is nonempty, that is, there exist some $p\in\{1,\ldots,m\}$ and $q\in \{1,\ldots,\ell-1\}$ such that  $a_p=b_q$. Noting that the function $\cos \theta$ is strongly monotonic in the interval $(0,\pi)$, we must have $\frac{p\pi}{m+1}=\frac{q\pi}{\ell}$, i.e., $p\ell=q(m+1)$. Thus, $(m+1)\mid p\ell$. But since we assume that $m+1$ and $\ell$  are coprime, we obtain that  $(m+1)\mid p$, contradicting the requirement that  $p\in \{1,\ldots,m\}$. This proves that $A\cap B=\emptyset$. Noting that $\gcd(m+1-\ell,m+1)=\gcd(m+1-\ell,\ell)=\gcd(\ell,m+1)=1$, similar arguments show that $A\cap C=B\cap C=\emptyset$ and hence Claim 1 follows.
  
  \noindent\textbf{Claim 2}: The polynomial $S_m(x)-\lambda S_{\ell-1}S_{m-\ell}(x)$ has only simple roots for any $\lambda\in \mathbb{R}$. 
  
  Write $T=B\cup C$ and label all the elements in decreasing order $t_1>t_2>\cdots>t_{m-1}$. Recall that $S_{\ell-1}S_{m-1}(x)$ is the characteristic polynomial of the vertex-deleted subgraph $P_m-\ell$. Thus, by the interlacing theorem and the established fact that $A\cap T=\emptyset$, we obtain 
  $a_1>t_1>a_2>t_2>\cdots>a_{m-1}>t_{m-1}>a_m$. Noting that each zero of $S_m(x)$ is simple, we find that the sequence $S_m(+\infty),S_m(t_1),S_m(t_2),\ldots,S_{m}(t_{m-1}),S_{m}(-\infty)$ must have alternating signs. Let $f(x)=S_m(x)-\lambda S_{\ell-1}(x)S_{m-\ell}(x)$. Since $S_{\ell-1}(t_i)S_{m-\ell}(t_i)=0$ and $m-1=\deg S_{\ell-1}(x)S_{m-\ell}(x)<\deg S_{m}(x)=m$, the sign of $f(x)$  coincides with that of $S_m(x)$ for each $x\in T\cup\{+\infty,-\infty\}$. It follows from the Intermediate Value Theorem  that $f(x)$ has at least one zero in each of the $m$ intervals $(-\infty,t_{m-1}), (t_{m-1},t_{m-2}),\ldots, (t_2,t_1)$ and  $(t_1,+\infty)$. Since $f(x)$ is a polynomial of degree $m$, we conclude that all zeroes of $f(x)$ are simple, completing the proof of Claim 2.
  
  Let $\lambda_1,\ldots,\lambda_n$ be the eigenvalues of $G$. By the assumption of this proposition, all these eigenvalues are pairwise different. By Claim 1, we find that $S_m(x)$ and $S_{\ell-1}(x)S_{m-\ell}(x)$ have no common roots and hence $\gcd(S_m(x),S_{\ell-1}(x)S_{m-\ell}(x))=1$. It follows that, for any two distinct $i,j\in \{1,2,\ldots,n\}$,  
  \begin{eqnarray*}
  &&\gcd(S_m(x)-\lambda_iS_{\ell-1}S_{m-\ell}(x),S_m(x)-\lambda_jS_{\ell-1}(x)S_{m-\ell}(x))\\
  	&=& \gcd(S_m(x)-\lambda_iS_{\ell-1}S_{m-\ell}(x),(\lambda_i-\lambda_j)S_{\ell-1}(x)S_{m-\ell}(x))\\
  		&=& \gcd(S_m(x),S_{\ell-1}(x)S_{m-\ell}(x))\\
  			&=& 1.
  \end{eqnarray*}
  Thus, the two polynomials $S_m(x)-\lambda_iS_{\ell-1}S_{m-\ell}(x)$ and $S_m(x)-\lambda_jS_{\ell-1}(x)S_{m-\ell}(x))$ have no common roots. This, combining with Claim 2, indicates that $\phi(G\circ P_m^{(\ell)})$ has only simple roots and hence completes the proof of Proposition \ref{nsc}.
  \end{proof}
  \begin{remark}\normalfont
  	For a graph $H$, Shan and Liu \cite{shan} call a vertex $v$ a Wronskian vertex  if $\gcd(\phi(H;x),\phi(H-v;x))=1$. In the same paper, they showed that $G\circ H^{(v)}$ has only simple eigenvalues if and only if $G$ has only simple eigenvalues and $v$ is a Wronskian vertex. Proposition \ref{nsc} gives a complete characterization of Wronskian vertices for any path graph $P_m$.
  	\end{remark}
  
\begin{corollary}
	Theorem \ref{main} holds when $G$ has a multiple root or $\gcd(\ell,m+1)>1$.
\end{corollary}
\begin{proof}
	It is known that any graph $H$ with a multiple eigenvalue must satisfy $\det W(H)=0$. Suppose that $G$ has a multiple root or $\gcd(\ell,m+1)>1$. It follows from Corollary \ref{chagpk} and Proposition \ref{nsc} that $G\circ P_{m}^{(\ell)}$ must have a multiple root and hence $\det W(G\circ P_{m}^{(\ell)})=0$. This completes the proof of this corollary.
\end{proof}
\begin{definition}\label{eigmu}\normalfont{
Let $\lambda_1,\ldots,\lambda_n$ denote the eigenvalues of $G$ and $\xi_1,\ldots,\xi_n$ be the corresponding  eigenvectors. 	We use $\mu_i^{(j)}(j\in\{1,2,\ldots,m\})$ to denote all zeroes of $S_m(x)-\lambda_iS_{\ell-1}(x)S_{m-\ell}(x)$ for  $i\in\{1,2,\ldots,n\}$ and write 
\begin{equation}\label{af}
	\eta_i^{(j)}=\left(S_{\ell-1}(\mu_i^{(j)})\begin{bmatrix}
		S_{m-1}(\mu_i^{(j)})\\
		\vdots\\
		S_{m-\ell+1}(\mu_i^{(j)})\\
			S_{m-\ell}(\mu_i^{(j)})\\
		\vdots\\
		S_0(\mu_i^{(j)})
	\end{bmatrix}-S_m(\mu_i^{(j)})\begin{bmatrix}
		S_{\ell-2}(\mu_i^{(j)})\\\vdots\\S_0(\mu_i^{(j)})\\0\\\vdots\\0
	\end{bmatrix}\right)\otimes \xi_i.
\end{equation}
}
\end{definition}

The main result of this section is the following.
\begin{lemma}\label{eigA}
	Let $\tilde{A}$ denote the adjacency matrix of $G\circ P_m^{(\ell)}$. Then $\tilde{A}\eta_i^{(j)}=\mu_i^{(j)}\eta_i^{(j)}$ for $i\in\{1,2,\ldots,n\}$ and $j\in\{1,2,\ldots,m\}$.
\end{lemma}
\begin{proof}
We fix $i$ and $j$ and write $\lambda=\lambda_i$, $\xi=\xi_i$, $\mu=\mu_i^{(j)}$, $\eta=\eta_i^{(j)}$ and $s_k=S_k(\mu_i^{(j)})$  for $k=-1,0,\ldots,m-1$, where we make the convention that $S_{-1}(x)=0$ and hence $s_{-1}=0$. According to \cite[Lemma 2.3]{wym2024},  Lemma \ref{eigA} is true for the special $\ell=1$. Thus, we may assume $\ell\ge 2$ in the following argument. By Observation  \ref{adjGP} and some basic properties of the Kronecker product, we obtain
	\begin{eqnarray}\label{Aeta}
\tilde{A}\eta &=&(A(P_m)\otimes I_n+D_\ell\otimes A(G))\left(\left(s_{\ell-1}\begin{bmatrix}
	s_{m-1}\\
	\vdots\\
	s_{m-\ell+1}\\
	s_{m-\ell}\\
	s_{m-\ell-1}\\
	\vdots\\
	s_0
\end{bmatrix}-s_m\begin{bmatrix}
s_{\ell-2}\\
\vdots\\
s_0\\
s_{-1}\\
0\\
\vdots\\
0
\end{bmatrix}\right)\otimes \xi\right)\nonumber\\
&=&\left(s_{\ell-1}A(P_m)\begin{bmatrix}
	s_{m-1}\\
	\vdots\\
	s_{m-\ell+1}\\
	s_{m-\ell}\\
	s_{m-\ell-1}\\
	\vdots\\
	s_0
\end{bmatrix}-s_mA(P_m)\begin{bmatrix}
s_{\ell-2}\\
\vdots\\
s_0\\
s_{-1}\\
0\\
\vdots\\
0
\end{bmatrix}+\begin{bmatrix}
0\\
\vdots\\
0\\
\lambda s_{\ell-1}s_{m-\ell}\\
0\\
\vdots\\
0
\end{bmatrix}\right)\otimes \xi \nonumber\\
&= &\left(s_{\ell-1}\begin{bmatrix}
s_{m-2}\\
s_{m-3}+s_{m-1}\\
\vdots\\
s_{m-\ell}+s_{m-\ell+2}\\
s_{m-\ell-1}+s_{m-\ell+1}\\
s_{m-\ell-2}+s_{m-\ell}\\
\vdots\\
s_0+s_2\\
s_{-1}+s_1
\end{bmatrix}-s_m\begin{bmatrix}
s_{\ell-3}\\
s_{\ell-4}+s_{\ell-2}\\
\vdots\\
s_{-1}+s_{1}\\
s_0\\
0\\
\vdots\\
0\\
0
\end{bmatrix}+\begin{bmatrix}
0\\
0\\
\vdots\\
0\\
\lambda s_{\ell-1}s_{m-\ell}\\
0\\
\vdots\\
0\\
0
\end{bmatrix}\right)\otimes \xi\label{bf}.
\end{eqnarray}
We denote the first factor in Eq.~\eqref{bf} by $(b_1,b_2,\ldots,b_m)^\T$ and the corresponding factor in Eq.~\eqref{af} by $(a_1,a_2,\ldots,a_m)^\T$. Noting that $s_{q}+s_{q+2}=\mu s_{q+1}$ for $q\ge -1$ by the three-term recurrence of $\{S_q(x)\}$, we find that 
\begin{equation}\label{bp}
	b_p=\begin{cases}
		s_{\ell-1}s_{m-2}-s_ms_{\ell-3},&p=1\\
		\mu s_{l-1}s_{m-p}-\mu s_{m}s_{\ell-p-1},&2\le p\le \ell-1\\
		\mu s_{\ell-1}s_{m-\ell}-s_ms_0+\lambda s_{\ell-1}s_{m-\ell},&p=\ell\\
		\mu s_{\ell-1}s_{m-p},&\ell+1\le p\le m.
\end{cases}\end{equation}
From Eq.~\eqref{af} and the fact that $s_{-1}=0$, we see that
\begin{equation}\label{ap}
	a_p=\begin{cases}
		s_{\ell-1}s_{m-p}-s_ms_{\ell-p-1},&1\le p\le \ell\\
		s_{\ell-1}s_{m-p},&\ell+1\le p\le m.
		\end{cases}
\end{equation}
To show $\tilde{A}\eta=\mu\cdot \eta$, it suffices to show $b_p=\mu a_p$ for each $p\in\{1,\ldots,m\}$.  Comparing Eq.~\eqref{bp} and Eq.~\eqref{ap}, we only need to check the equality  $b_p=\mu a_p$ for two cases: $p=1$ and $p=\ell$.

Direct calculation shows that 
\begin{eqnarray*}
	b_1&=&s_{\ell-1}s_{m-2}-s_{m}s_{\ell-3}\\&=&s_{\ell-1}(s_{m-2}+s_m)-s_m(s_{\ell-3}+s_{\ell-1})\\
	&=&\mu s_{\ell-1}s_{m-1}-\mu s_ms_{l-2}\\
	&=&\mu a_1.
\end{eqnarray*}
It remains to show $b_p=\mu a_{p}$ for $p=\ell$. Recalling that $\mu$ is a zero of $S_m(x)-\lambda S_{\ell-1}(x)S_{m-\ell}(x)$, i.e., $s_m=\lambda s_{\ell-1}s_{m-\ell}$, we obtain
\begin{equation*}
	b_\ell=\mu s_{\ell-1}s_{m-\ell}-s_ms_0+\lambda s_{\ell-1}s_{m-\ell}=\mu s_{\ell-1}s_{m-\ell}.
\end{equation*}
Since $s_{-1}=0$, we see from Eq.~\eqref{ap} that  $a_\ell=s_{\ell-1}s_{m-p}$. Thus	$b_\ell=\mu a_\ell$, as desired. This completes the proof of Lemma \ref{eigA}.
\end{proof}
\section{Computing $\det W(G\circ P_m^{(\ell)})$}\label{cdw}
In this section, we assume  that all eigenvalues $\lambda_1,\ldots,\lambda_n$ are simple. Also assume $2\le \ell\le  (m+1)/2$ and $\gcd(\ell,m)=1$.

 Our basic tool is the  following formula to compute $\det W(G)$   using the  eigenvalues and eigenvectors of $G$. 
\begin{lemma}[\cite{mao2015}]\label{basicW}
	Let $\lambda_i$ be the eigenvalues of $G$ with  eigenvector $\xi_i$ for $i=1,2,\ldots,n$. Then
	$$\det W(G)= \frac{\prod_{1\le i_1< i_2\le n}(\lambda_{i_2}-\lambda_{i_1})\prod_{1\le i\le n}(e_n^\T \xi_i)}{\det[\xi_1,\xi_2,\ldots,\xi_n]}.$$
\end{lemma}
Let $\Omega=\{(i,j)\colon\,1\le i\le n \text{~and~} 1\le j\le m\}$ with the colexicographical order: $(i_1,j_1)<(i_2,j_2) $ if either $j_1<j_2$, or $j_1=j_2$ and $i_1<i_2$. By Proposition \ref{nsc}, the graph $G\circ P_m^{(\ell)}$ has only simple eigenvalues. The following formula of $\det W(G\circ P_m^{(\ell)})$ is an immediate consequence of Lemma \ref{eigA}  and Lemma \ref{basicW}.
\begin{corollary}\label{dwt3}
	\begin{equation*}
		\det W(G\circ P_m^{(\ell)})= \frac{\prod_{(i_1,j_1)<(i_2,j_2)}(\mu_{i_2}^{(j_2)}-\mu_{i_1}^{(j_1)})\prod_{(i,j)\in \Omega}(e_{mn}^\T \eta_i^{(j)})}{\det[\eta_1^{(1)},\ldots,\eta_n^{(1)};\ldots;\eta_1^{(m)},\ldots,\eta_n^{(m)}]}.
	\end{equation*}
\end{corollary}

\begin{definition}\label{dS}
	$S(x)=S_{\ell-1}(x)\sum_{k=0}^{m-1}S_k(x)-S_m(x)\sum_{k=0}^{\ell-2}S_k(x)$.
\end{definition}
The main aim of this section is to establish the following connection between the two determinants $\det W(G\circ P_{m}^{(\ell)})$ and $\det W(G)$.
\begin{proposition}\label{dwt}
	\begin{equation*}
	\det W(G\circ P_m^{(\ell)})=\pm \left(\prod_{1\le i_1<i_2\le n}\prod_{j_2=1}^{m}S_{\ell-1}(\mu_{i_2}^{(j_2)})S_{m-\ell}(\mu_{i_2}^{(j_2)})\right)\left(\prod_{i=1}^n\prod_{j=1}^m S(\mu_i^{(j)})\right)\left(\det W(G)\right)^m.
	\end{equation*}
\end{proposition}

\begin{lemma}\label{ppmu}
	\begin{equation*}
	\prod_{j_2=1}^{m}\prod_{j_1=1}^{m}\left(\mu_{i_2}^{(j_2)}-\mu_{i_1}^{(j_1)}\right)=(\lambda_{i_2}-\lambda_{i_1})^m\prod_{j_2=1}^{m}S_{\ell-1}(\mu_{i_2}^{(j_2)})S_{m-\ell}(\mu_{i_2}^{(j_2)}).
	\end{equation*}
\end{lemma}
\begin{proof}
Note that $S_m(x)-\lambda_iS_{\ell-1}(x)S_{m-\ell}(x)$ is monic and has roots $\mu_i^{(1)},\ldots,\mu_i^{(m)}$, we have
$S_m(x)-\lambda_iS_{\ell-1}(x)S_{m-\ell}(x)=\prod_{j_1=1}^m(x-\mu_i^{(j_1)}).$ Thus,
	\begin{equation*}
S_m(\mu_{i_2}^{(j_2)})-\lambda_{i_1}S_{\ell-1}(\mu_{i_2}^{(j_2)})S_{m-\ell}(\mu_{i_2}^{(j_2)})=\prod_{j_1=1}^{m}\left(\mu_{i_2}^{(j_2)}-\mu_{i_1}^{(j_1)}\right).
\end{equation*}
This, together with the fact that $S_m(\mu_{i_2}^{(j_2)})-\lambda_{i_2}S_{\ell-1}(\mu_{i_2}^{(j_2)})S_{m-\ell}(\mu_{i_2}^{(j_2)})=0$ implies
\begin{eqnarray}
	\prod_{j_2=1}^{m}\prod_{j_1=1}^{m}\left(\mu_{i_2}^{(j_2)}-\mu_{i_1}^{(j_1)}\right)&=&\prod_{j_2=1}^{m}\left(S_m(\mu_{i_2}^{(j_2)})-\lambda_{i_1}S_{\ell-1}(\mu_{i_2}^{(j_2)})S_{m-\ell}(\mu_{i_2}^{(j_2)})\right)\nonumber\\
&=&(\lambda_{i_2}-\lambda_{i_1})^m\prod_{j_2=1}^{m}S_{\ell-1}(\mu_{i_2}^{(j_2)})S_{m-\ell}(\mu_{i_2}^{(j_2)})\nonumber.
\end{eqnarray}
 This completes the proof.
\end{proof}
\begin{lemma}\label{Vanmu}
	\begin{equation*}
\prod_{(i_1,j_1)<(i_2,j_2)}(\mu_{i_2}^{(j_2)}-\mu_{i_1}^{(j_1)})=\pm	\left(\prod_{i=1}^{n}\prod_{1\le j_1<j_2\le m}\left(\mu_i^{(j_2)}-\mu_i^{(j_1)}\right)\right)\left(\prod_{1\le i_1< i_2\le n}	\prod_{j_2=1}^{m}\prod_{j_1=1}^{m}\left(\mu_{i_2}^{(j_2)}-\mu_{i_1}^{(j_1)}\right)\right).
	\end{equation*}
\end{lemma}
\begin{proof}
	We have 
	\begin{equation*}\label{pmumu}
		\prod_{(i_1,j_1)<(i_2,j_2)}(\mu_{i_2}^{(j_2)}-\mu_{i_1}^{(j_1)})=\left(\prod_{i=1}^{n}\prod_{1\le j_1<j_2\le m}\left(\mu_{i}^{(j_2)}-\mu_{i}^{(j_1)}\right)\right)\left(\prod_{i_1\neq i_2}\prod_{(i_1,j_1)<(i_2,j_2)}\left(\mu_{i_2}^{(j_2)}-\mu_{i_1}^{(j_1)}\right)\right).
	\end{equation*}
	The second factor can be regrouped as 
	\begin{eqnarray}\label{sf}
		&&\prod_{1\le i_1< i_2\le n}\left(\prod_{(i_1,j_1)<(i_2,j_2)}\left(\mu_{i_2}^{(j_2)}-\mu_{i_1}^{(j_1)}\right)\right)\left(\prod_{(i_2,j_2)<(i_1,j_1)}\left(\mu_{i_1}^{(j_1)}-\mu_{i_2}^{(j_2)}\right)\right)\nonumber\\
		&=&\prod_{1\le i_1< i_2\le n}\left(\prod_{j_2=1}^{m}\prod_{j_1=1}^{m}\left(\mu_{i_2}^{(j_2)}-\mu_{i_1}^{(j_1)}\right)\right)\left(\prod_{(i_2,j_2)<(i_1,j_1)}(-1)\right)\nonumber\\
		&=&\pm \prod_{1\le i_1< i_2\le n}\prod_{j_2=1}^{m}\prod_{j_1=1}^{m}\left(\mu_{i_2}^{(j_2)}-\mu_{i_1}^{(j_1)}\right)\nonumber.
	\end{eqnarray}
This proves Lemma \ref{Vanmu}.
\end{proof}

\begin{lemma}\label{exm}
	$\prod_{(i,j)\in \Omega}(e_{mn}^\T \eta_i^{(j)})= \left(\prod_{(i,j)\in \Omega}S(\mu_i^{(j)})\right)\left(\prod_{1\le i\le n}e_n^\T \xi_i\right)^m.$
\end{lemma}
\begin{proof}
	Noting that $e_{mn}^\T=(e_m)^\T \otimes (e_n)^\T$ together with the definitions of $\eta_i^{(j)}$ and $S(x)$, we have
	\begin{eqnarray}
	e_{mn}^\T\eta_i^{(j)}&=&(e_m^\T\otimes e_n^\T)\left(\left(S_{\ell-1}(\mu_i^{(j)})\begin{bmatrix}
		S_{m-1}(\mu_i^{(j)})\\
		\vdots\\
		S_{m-\ell+1}(\mu_i^{(j)})\\
		S_{m-\ell}(\mu_i^{(j)})\\
		\vdots\\
		S_0(\mu_i^{(j)})
	\end{bmatrix}-S_m(\mu_i^{(j)})\begin{bmatrix}
		S_{\ell-2}(\mu_i^{(j)})\\\vdots\\S_0(\mu_i^{(j)})\\0\\\vdots\\0
	\end{bmatrix}\right)\otimes \xi_i\right)\nonumber\\
	&=&\left(S_{\ell-1}(\mu_i^{(j)})\sum_{k=0}^{m-1}S_k(\mu_i^{(j)})-S_{m}(\mu_i^{(j)})\sum_{k=0}^{\ell-2}S_k(\mu_{i}^{(j)})\right)e_n^\T\xi_i\nonumber\\
	&=&S(\mu_i^{(j)})e_n^\T\xi_i.\nonumber
	\end{eqnarray}
	Noting that  $\prod_{(i,j)\in \Omega}(e_n^\T \xi_i)=\prod_{j=1}^{m}\prod_{i=1}^n e_n^\T\xi_i=\left(\prod_{i=1}^n e_n^\T\xi_i\right)^m$, Lemma \ref{exm} follows.
\end{proof}

\begin{lemma}\label{deteta1}
	\begin{equation*}
		\det[\eta_1^{(1)},\ldots,\eta_n^{(1)};\ldots;\eta_1^{(m)},\ldots,\eta_n^{(m)}]= \left(\det[\xi_1,\xi_2,\ldots,\xi_n]\right)^m \prod_{i=1}^{n}\det (f_{k}(\mu_i^{(j)}))_{m\times m},
	\end{equation*}
	where \begin{equation}\label{fk}
		f_{k}(x)=\begin{cases}
			S_{\ell-1}(x)S_{m-k}(x)-S_{m}(x)S_{\ell-k-1}(x),&1\le k\le \ell-1\\
			S_{\ell-1}(x)S_{m-k}(x),&\ell\le k\le m,
			\end{cases}
			\end{equation}
 and $(f_{k}(\mu_i^{(j)}))$ is an $m\times m$ matrix whose $(k,j)$-entry is $f_{k}(\mu_i^{(j)})$. 
\end{lemma}
\begin{proof}
	Let $E^{(j)}=[\eta_1^{(j)},\eta_2^{(j)},\ldots,\eta_n^{(j)}]$ for $j\in \{1,\ldots,m\}$. By Definition \ref{eigmu}, we have
		\begin{eqnarray}\label{ej}
		E^{(j)}&=&\left[\begin{bmatrix}
			f_{1}(\mu_1^{(j)})\cdot \xi_1\\
			f_{2}(\mu_1^{(j)}) \cdot\xi_1\\
			\vdots\\
			f_{m}(\mu_1^{(j)})\cdot \xi_1
		\end{bmatrix},\begin{bmatrix}
			f_{1}(\mu_2^{(j)}) \cdot\xi_2\\
			f_{2}(\mu_2^{(j)})\cdot\xi_2\\
			\vdots\\
			f_{m}(\mu_2^{(j)}) \cdot\xi_2
		\end{bmatrix},\cdots,\begin{bmatrix}
			f_{1}(\mu_n^{(j)}) \cdot\xi_n\\
			f_{2}(\mu_n^{(j)}) \cdot\xi_n\\
			\vdots\\
			f_{m}(\mu_n^{(j)})\cdot\xi_n
		\end{bmatrix}\right]\nonumber\\
		&=&\begin{bmatrix}
			[\xi_1,\xi_2\,\ldots,\xi_n]\cdot\diag[	f_{1}(\mu_1^{(j)}),		f_{1}(\mu_2^{(j)}),\ldots,		f_{1}(\mu_n^{(j)})]\\
			[\xi_1,\xi_2\,\ldots,\xi_n]\cdot\diag[	f_{2}(\mu_1^{(j)}),		f_{2}(\mu_2^{(j)}),\ldots,		f_{2}(\mu_n^{(j)}) ]\\
			\vdots\\
			[\xi_1,\xi_2\,\ldots,\xi_n]\cdot\diag[	f_{m}(\mu_1^{(j)}),	f_{m}(\mu_2^{(j)}),\ldots,		f_{m}(\mu_n^{(j)})]
		\end{bmatrix}.
	\end{eqnarray}
	
	Let $Q=[\xi_1,\xi_2,\ldots,\xi_n]$ and $T_{k,j}=\diag[f_{k}(\mu_1^{(j)}),	f_{k}(\mu_2^{(j)}),\ldots,	f_{k}(\mu_n^{(j)})]$ for $k,j\in\{1,2,\ldots,m\}$. Then we can rewrite Eq.~\eqref{ej} as 
	\begin{equation*}
	E^{(j)}=\begin{bmatrix} Q\cdot T_{1,j}\\
	 Q\cdot T_{2,j}\\
	 \vdots\\
	  Q\cdot T_{m,j}
	  \end{bmatrix}=\begin{bmatrix}Q&&&\\
	  &Q&&\\
	  &&\ddots&\\
	  &&&Q\end{bmatrix}\begin{bmatrix}  T_{1,j}\\
	   T_{2,j}\\
	  \vdots\\
	   T_{m,j}
	  \end{bmatrix}.
	\end{equation*}
which implies the following identity:
\begin{equation}\label{e1n}
[E^{(1)},E^{(2)},\ldots,E^{(m)}]=\begin{bmatrix}
Q&&&\\
&Q&&\\
&&\ddots&\\
&&&Q
\end{bmatrix} \begin{bmatrix}
T_{1,1}&T_{1,2}&\cdots&T_{1,m}\\
T_{2,1}&T_{2,2}&\cdots&T_{2,m}\\
\vdots&\vdots&&\vdots\\
T_{m,1}&T_{m,2}&\cdots&T_{m,m}
\end{bmatrix}.
\end{equation}
Since each $T_{k,j}$ is a diagonal matrix, one easily sees that the block matrix $(T_{k,j})$ is permutationally similar to the following block diagonal matrix
\begin{equation*}
\begin{bmatrix}
\begin{bmatrix}
	f_{1}(\mu_{1}^{(1)})&	f_{1}(\mu_{1}^{(2)})&\ldots& 	f_{1}(\mu_{1}^{(m)})\\
	f_{2}(\mu_{1}^{(1)})&	f_{2}(\mu_{1}^{(2)})&\ldots& 	f_{2}(\mu_{1}^{(m)})\\
	\vdots&\vdots&&\vdots\\
	f_{m}(\mu_{1}^{(1)})&	f_{m}(\mu_{1}^{(2)})&\ldots& 	f_{m}(\mu_{1}^{(m)})\\
\end{bmatrix}&&&\\
&\ddots&\\
&&&\begin{bmatrix}
	f_{1}(\mu_{n}^{(1)})&f_{1}(\mu_{n}^{(2)})&\ldots& 	f_{1}(\mu_{n}^{(m)})\\
		f_{2}(\mu_{n}^{(1)})&f_{2}(\mu_{n}^{(2)})&\ldots& 	f_{2}(\mu_{n}^{(m)})\\
	\vdots&\vdots&&\vdots\\
	f_{m}(\mu_{n}^{(1)})&f_{m}(\mu_{n}^{(2)})&\ldots& 	f_{m}(\mu_{n}^{(m)})\\
\end{bmatrix}
\end{bmatrix},
\end{equation*}
or written more compactly, the block matrix
\begin{equation*}
	\diag[(f_{k}(\mu_1^{(j)}))_{m\times m},(f_{k}(\mu_2^{(j)}))_{m\times m},\ldots,(f_{k}(\mu_n^{(j)}))_{m\times m}].
\end{equation*}
Thus, taking determinants on both sides of Eq.~\eqref{e1n} leads to
\begin{equation*}
\det [E^{(1)},E^{(2)},\ldots,E^{(m)}]=(\det Q)^m\prod_{1\le i\le n}\det(f_{k}(\mu_i^{(j)}))_{m\times m}.
\end{equation*}
This completes the proof.
\end{proof}
A $(0,1)$-matrix $B$ has the \emph{consecutive ones property} (for rows) if there exists a permutation of its columns that leaves the 1's consecutive in every row. It is well known that such matrices are totally unimodular \cite{fulkerson}, that is, every square submatrix has determinant $0, +1$ or $-1$. In particular, every nonsingular $(0,1)$-matrix with the consecutive ones property  has determinant $\pm 1$.

\begin{lemma}\label{cons}
	Let $f_k(x)$ with $1\le k\le m$ be defined as in Eq.~\eqref{fk}. Then there exists a $(0,1)$-matrix $B$ with consecutive ones property such that 
	\begin{equation*}
		\begin{bmatrix}
		f_1(x)\\f_2(x)\\\vdots\\f_m(x)
		\end{bmatrix}=B\begin{bmatrix}
		S_{0}(x)\\S_{1}(x)\\\vdots\\S_{m-1}(x)
		\end{bmatrix}.
	\end{equation*}
\end{lemma}
\begin{proof}
	For two nonnegative integers $a$ and $b$ with $a\le b$ and $a\equiv b\pmod{2}$, we define $\mathcal{I}(a,b)=\{i\colon\, a\le i\le b \text{~and~} i\equiv a\pmod{2}\}$. 
	
	\noindent{\textbf{Claim 1 }: We have $S_p(x)S_q(x)=\sum_{i}S_i(x)$ for $p\ge q\ge 0$, where the summation is taken over $\mathcal{I}(p-q,p+q)$.}
	
	Let $x=\cos \theta$. Recalling that $S_i(\cos \theta) =\frac{\sin(i+1)\theta}{\sin \theta}$, we have
	\begin{eqnarray*}
		 \sum_{i\in \mathcal{I}(p-q,p+q)}S_i(\cos\theta)&=& \sum_{j=0}^{q} S_{p-q+2j}(\cos\theta)\nonumber\\
		 &=& \frac{1}{\sin\theta}\sum_{j=0}^{q}\sin(p-q+2j+1)\theta\nonumber\\
		  &=& \frac{1}{\sin^2\theta}\sum_{j=0}^{q}\sin(p-q+2j+1)\theta\cdot\sin\theta\nonumber\\
		  &=&-\frac{1}{2\sin^2\theta}\sum_{j=0}^{q}(\cos(p-q+2j+2)\theta-\cos(p-q+2j)\theta)\nonumber\\
		  &=&-\frac{1}{2\sin^2\theta}(\cos(p-q+2q+2)\theta-\cos(p-q)\theta)\nonumber\\
		  &=&\frac{1}{\sin^2\theta}(\sin(p+1)\theta\cdot\sin(q+1)\theta)\nonumber\\
		  &=&S_p(\cos\theta)S_q(\cos\theta).
	\end{eqnarray*}
	This proves Claim 1.
	
	\noindent{\textbf{Claim 2}: We have
		\begin{equation}\label{fk2}
			f_{k}(x)=\begin{cases}
				\sum\limits_{i\in \mathcal{I}_k}S_i(x),&1\le k\le \ell-1\\
				\sum\limits_{i\in \mathcal{I}_k}S_i(x),&\ell\le k\le m+1-\ell\\
			\sum\limits_{i\in \mathcal{I}_k}S_i(x),&m+2-\ell\le k\le m,
			\end{cases}
			\end{equation}
			where \begin{equation}\label{Ik}
			\mathcal{I}_k=\begin{cases}
				\mathcal{I}(m-\ell-k+1,m-\ell+k-1),&1\le k\le \ell-1\\
				\mathcal{I}(m-\ell-k+1,m+\ell-k-1),&\ell\le k\le m+1-\ell\\
				\mathcal{I}(\ell+k-m-1,\ell-k+m-1),&m+2-\ell\le k\le m.
			\end{cases}
			\end{equation}
			Moreover, $\mathcal{I}_k\subset\{0,1,\ldots,m-1\}$ for $k=1,\ldots,m$.}
			
Let $1\le k\le \ell-1$.   Note that $f_k(x)=S_{\ell-1}(x)S_{m-k}(x)-S_m(x)S_{\ell-k-1}$. As $2\le \ell\le (m+1)/2$ and $1\le k\le \ell-1$, we easily find that $m-k\ge \ell-1$ and $m\ge \ell-k-1$. It follows from Claim 1 that
\begin{equation*}
	f_k(x)=\sum_{i\in\mathcal{I}(m-\ell-k+1,m+\ell-k-1)}S_i(x)-\sum_{i\in\mathcal{I}(m-\ell+k+1,m+\ell-k-1)}S_i(x)=\sum_{i\in\mathcal{I}(m-\ell-k+1,m-\ell+k-1)}S_i(x).
\end{equation*}
This proves  Claim 2 for the case $1\le k\le \ell-1$. For the case $\ell\le k\le m+1-\ell$, we have $m-k\ge \ell-1$ and hence 
\begin{equation*}
	f_k(x)=\sum_{i\in\mathcal{I}(m-\ell-k+1,m+\ell-k-1)}S_i(x),
\end{equation*}
as desired.  Similarly, for the case $m+2-\ell\le k\le m$, we have $ \ell-1\ge m-k$ and hence 
\begin{equation*}
	f_k(x)=S_{m-k}(x)S_{\ell-1}(x)=\sum_{i\in\mathcal{I}(\ell+k-m-1,\ell-k+m-1)}S_i(x).
\end{equation*}
Finally, by comparing the size relationship between  $k$ and $\ell$ in each case, we find that the maximum value in each $\mathcal{I}_k$ is at most $m-1$. Thus, $\mathcal{I}_k\subset\{0,1,\ldots,m-1\}$ always holds. This proves Claim 2.

Let $B=(b_{kj})$ be the (0,1)-matrix corresponding to $\{\mathcal{I}_k\}$ in the sense that $b_{k,j+1}=1$ if and only if $j\in \mathcal{I}_k$. Then we can write Eq.~\eqref{fk2} in the matrix form: 
	\begin{equation*}
	\begin{bmatrix}
		f_1(x)\\f_2(x)\\\vdots\\f_m(x)
	\end{bmatrix}=B\begin{bmatrix}
		S_{0}(x)\\S_{1}(x)\\\vdots\\S_{m-1}(x)
	\end{bmatrix}.
\end{equation*}
We claim that $B$ has the consecutive ones property. Write the $j$-th column of $B$ as $\beta_j$ for $j=1,\ldots,m$. We construct a new matrix $B'$ as follows:
\begin{equation*}
	B'=[\beta_1,\beta_3,\ldots,\beta_{p};\beta_2,\beta_4,\ldots,\beta_{q}],
\end{equation*}
where $p$ (resp. $q$) is the largest odd (resp. even) integer not exceeding $m$. In other words,  $B'$ is obtained from $B$ by permuting columns such that columns with odd indices are placed first, followed by columns with even indices. Note that for any fixed $k$, the set $\mathcal{I}_k$  consists of either consecutive even numbers or  consecutive odd numbers. It follows that the 1s in each row of the resulting matrix $B'$ are consecutive. This means that $B$ has the consecutive ones property and hence completes the  proof of Lemma \ref{cons}.
\end{proof}
\begin{remark}\label{rmS}\normalfont{
	It is easy to see from Eq.~\eqref{Ik} that $m-1\in \mathcal{I}_k$ if and only if $k=\ell$. Noting that $S_k(x)$ is a monic polynomial of degree $k$, we find from Eq.~\eqref{fk2} that $\sum_{k=1}^mf_k(x)$ (i.e., $S(x)$) is a monic polynomial of degree $m-1$.}
\end{remark}
An important application of Lemma \ref{cons} is to determine (up to a sign) the number $\det(f_k(\mu_i^{(j)}))$ appearing in Lemma \ref{deteta1}. 
\begin{corollary}\label{van}
	$\det (f_k(\mu_i^{(j)}))_{m\times m}=\pm \prod_{1\le j_1< j_2\le m}\left(\mu_i^{(j_2)}-\mu_i^{(j_1)}\right)$ for $i=1,\ldots,n$. 
\end{corollary}
\begin{proof}
	Let $B$ be the (0,1)-matrix stated in Lemma \ref{cons}. Then, substituting $x$ with $\mu_{i}^{(j)}$, we have 
	\begin{equation*}
	\begin{bmatrix}
		f_1(\mu_i^{(j)})\\f_2(\mu_i^{(j)})\\\vdots\\f_m(\mu_i^{(j)})
	\end{bmatrix}=B\begin{bmatrix}
		S_{0}(\mu_i^{(j)})\\S_{1}(\mu_i^{(j)})\\\vdots\\S_{m-1}(\mu_i^{(j)})
	\end{bmatrix} \text{~for~} j=1,\ldots,m.
	\end{equation*}	
	Written in matrix form, we obtain
		\begin{equation*}
		\begin{bmatrix}
			f_1(\mu_i^{(1)})&f_1(\mu_i^{(2)})&\cdots&f_1(\mu_i^{(m)})\\
			f_2(\mu_i^{(1)})&f_2(\mu_i^{(2)})&\cdots&f_2(\mu_i^{(m)})\\
			\vdots&\vdots&&\vdots\\			
			f_m(\mu_i^{(1)})&f_m(\mu_i^{(2)})&\cdots&f_m(\mu_i^{(m)})
		\end{bmatrix}=B\begin{bmatrix}
			S_{0}(\mu_i^{(1)})&S_{0}(\mu_i^{(2)})&\cdots&S_{0}(\mu_i^{(m)})\\
			S_{1}(\mu_i^{(1)})&S_{1}(\mu_i^{(1)})&\cdots&S_{1}(\mu_i^{(m)})\\
				\vdots&\vdots&&\vdots\\	
				S_{m-1}(\mu_i^{(1)})&S_{m-1}(\mu_i^{(1)})&\cdots&S_{m-1}(\mu_i^{(m)})
		\end{bmatrix}.
	\end{equation*}	
	Taking determinants on both sides, we obtain 	$\det (f_k(\mu_i^{(j)}))=\det B\det (S_{k-1}(\mu_i^{(j)}))$. Recalling  that 	$\det[\eta_1^{(1)},\ldots,\eta_n^{(1)};\ldots;\eta_1^{(m)},\ldots,\eta_n^{(m)}] \neq 0$ as $G\circ P_m^{(\ell)}$ has only simple roots, we see from Lemma \ref{deteta1} that 	$\det (f_k(\mu_i^{(j)}))\neq 0$ and hence $\det B\neq 0$. Consequently, $\det B=\pm 1$ as $B$ has the consecutive ones property. This means that
		$\det (f_k(\mu_i^{(j)}))=\pm \det (S_{k-1}(\mu_i^{(j)}))$.  Since $S_{k-1}(x)$ is a monic polynomial with degree $k-1$ for $k=1,2,\ldots,m$, there exists a unit lower triangular matrix $L$ such that 
			\begin{equation*}
			\begin{bmatrix}
				S_0(x)\\S_1(x)\\\vdots\\S_{m-1}(x)
			\end{bmatrix}=L\begin{bmatrix}
				1\\x\\\vdots\\x^{m-1}
			\end{bmatrix}.
		\end{equation*}
	It follows that $\det (S_{k-1}(\mu_i^{(j)}))$ equals the Vandermonde determinant 
	\begin{equation*}
		\det\begin{bmatrix}
			1&1&\cdots&1\\
			\mu_i^{(1)}&\mu_i^{(2)}&\cdots&\mu_i^{(1)}\\
			\vdots&\vdots&&\vdots\\
			\mu_i^{(m-1)}&\mu_i^{(m-1)}&\cdots&\mu_i^{(m-1)}
		\end{bmatrix},
	\end{equation*}
	which equals $\prod_{1\le j_1< j_2\le m}\left(\mu_i^{(j_2)}-\mu_i^{(j_1)}\right)$. This completes the proof of Corollary \ref{van}.	
\end{proof}
Now we are in a position to provide a proof of  Proposition \ref{dwt}.

\noindent\textbf{Proof of Proposition \ref{dwt}} For convenience, let us denote
\begin{equation}\label{d1}
	\Delta_1=\prod_{1\le i_1< i_2\le n}\prod_{j_2=1}^{m}S_{\ell-1}(\mu_{i_2}^{(j_2)})S_{m-\ell}(\mu_{i_2}^{(j_2)})
\end{equation}
\begin{equation}\label{d2}
\Delta_2=\prod_{(i,j)\in\Omega} S(\mu_i^{(j)})=\prod_{i=1}^n\prod_{j=1}^m S(\mu_i^{(j)})
\end{equation} 
and 
\begin{equation*}
\Delta_3=\prod_{i=1}^{n}\prod_{1\le j_1<j_2\le m}\left(\mu_i^{(j_2)}-\mu_i^{(j_1)}\right).
\end{equation*}
 By  Lemmas \ref{ppmu} and \ref{Vanmu}, we have
 
 	\begin{equation*}
 	\prod_{(i_1,j_1)<(i_2,j_2)}(\mu_{i_2}^{(j_2)}-\mu_{i_1}^{(j_1)})=\pm	\Delta_3\cdot\left(\prod_{1\le i_1< i_2\le n}(\lambda_{i_2}-\lambda_{i_1})^m\right)\cdot\Delta_1.
 \end{equation*}
 Combining Lemma \ref{deteta1} and Corollary \ref{van} leads to 
  \begin{equation*}
 	\det[\eta_1^{(1)},\ldots,\eta_n^{(1)};\ldots;\eta_1^{(m)},\ldots,\eta_n^{(m)}]=\pm  \left(\det[\xi_1,\xi_2,\ldots,\xi_n]\right)^m\cdot\Delta_3.
 \end{equation*}
 It follows from Corollary \ref{dwt3} and Lemma \ref{exm} that
 \begin{eqnarray*}
&&\det W(G\circ P_m^{(\ell)})\\
&=& \frac{\prod_{(i_1,j_1)<(i_2,j_2)}(\mu_{i_2}^{(j_2)}-\mu_{i_1}^{(j_1)})\prod_{(i,j)\in \Omega}(e_{mn}^\T \eta_i^{(j)})}{\det[\eta_1^{(1)},\ldots,\eta_n^{(1)};\ldots;\eta_1^{(m)},\ldots,\eta_n^{(m)}]}\\
&=&\pm \frac{\left(\Delta_3\left(\prod_{1\le i_1< i_2\le n}(\lambda_{i_2}-\lambda_{i_1})^m\right)\Delta_1\right)\cdot \left(\Delta_2\left(\prod_{1\le i\le n}e_n^\T \xi_i\right)^m\right)}{ \left(\prod_{1\le i\le n}e_n^\T \xi_i\right)^m\left(\det[\xi_1,\xi_2,\ldots,\xi_n]\right)^m\cdot\Delta_3}\\
&=&\pm \Delta_1\Delta_2 \left(\frac{\prod_{1\le i_1< i_2\le n}(\lambda_{i_2}-\lambda_{i_1})\prod_{1\le i\le n}(e_n^\T \xi_i)}{\det[\xi_1,\xi_2,\ldots,\xi_n]}\right)^m\\
&=&\pm \Delta_1\Delta_2(\det W(G))^m.
\end{eqnarray*}
This completes the proof of Proposition \ref{dwt}.

\section{Proof of Theorem \ref{main}}\label{pft}
According to Proposition \ref{dwt}, to show Theorem \ref{main}, we need to determine the exact values of $\Delta_1$ and $\Delta_2$  defined in Eqs.~\eqref{d1} and \eqref{d2}. 
\begin{definition}\normalfont{
 		Let $f(x)=a_nx^n+a_{n-1}x^{n-1}+\cdots+a_1x+a_0$ and $g(x)=b_mx^m+b_{m-1}x^{m-1}+\cdots+b_1x+b_0$. The resultant of $f(x)$ and $g(x)$, denoted by  $\Res(f(x),g(x))$, is defined to be
 		$$a_n^mb_m^n\prod_{1\le i\le n,1\le j\le m}(\alpha_i-\beta_j),$$
 		where $\alpha_i$'s and $\beta_j$'s are the roots (in complex field $\mathbb{C}$) of $f(x)$ and $g(x)$, respectively.
 	}
 \end{definition}
 We list some basic properties of the resultants for convenience.
 \begin{lemma}\label{basicres}
 	Let $f(x)=a_nx^n+\cdots+a_0=a_n\prod_{i=1}^n(x-\alpha_i)$ and $g(x)=b_mx^m+\cdots+b_0=b_m\prod_{j=1}^m(x-\beta_j)$. Then the following statements hold:\\
 	\textup{(\rmnum{1})} $\Res(f(x),g(x))=a_n^m\prod_{i=1}^{n}g(\alpha_i) =(-1)^{mn}b_m^{n}\prod_{j=1}^m f(\beta_j);$\\
 	\textup{(\rmnum{2})} If $m<n$ then $\Res(f(x)+tg(x),g(x))=\Res(f(x),g(x))$ for any $t\in \mathbb{C}$;\\
 	\textup{(\rmnum{3})} $\Res(f(tx+s),g(tx+s))=\Res(f(tx),g(tx))=t^{mn}\Res(f(x),g(x))$ for any $t\in \mathbb{C}\setminus\{0\}$ and $s\in \mathbb{C}$.
 \end{lemma}
 
 \begin{lemma} \label{res1}
 For any integers $\ell, m$ with $2\le \ell\le (m+1)/2$ and $\gcd(\ell,m+1)=1$,
 $$	\Res\left(U_m(x),U_{\ell-1}(x)U_{m-\ell}(x)\right)=\pm 2^{m(m-1)}.$$
 \end{lemma}
\begin{proof}
Note that the zeroes of $U_{\ell-1} (x)U_{m-\ell}(x)$ are $\cos\frac{\pi j}{\ell}$ for $1\le j\le \ell-1$ and $\cos\frac{\pi j}{m+1-\ell}$ for $1\le j\le m-\ell$. As $U_m(x)$ is of degree $m$ and $U_{\ell-1}(x)U_{m-\ell}(x)$ is of degree $m-1$ with leading coefficient $2^{m-1}$, the resultant factors as 
\begin{equation}
	(-1)^{m(m-1)}2^{m(m-1)}\prod_{j=1}^{\ell-1}U_m\left(\cos\frac{\pi j}{\ell}\right)\prod_{j=1}^{m-\ell}U_m\left(\cos\frac{\pi j}{m+1-\ell}\right)
\end{equation}
We claim that $\prod_{j=1}^{\ell-1}U_m\left(\cos\frac{\pi j}{\ell}\right)=\pm 1$. Indeed, we have 
\begin{equation}
\prod_{j=1}^{\ell-1}	U_m\left(\cos\frac{\pi j}{\ell}\right)=\prod_{j=1}^{\ell-1}\frac{\sin\frac{\pi (m+1)j} \ell}{\sin \frac{\pi j}{\ell}}
\end{equation}
For $j\in\{1,2,\ldots,\ell-1\}$, let $r_j$ be the least 
nonnegative residue of $(m+1)j$ modulo $\ell$. As $\gcd(\ell,m+1)=1$, we easily see that $r_j\neq 0$ and all these $r_j$'s are pairwise different. This means that $(r_1,r_2,\ldots,r_{\ell-1})$ is a permutation of $\{1,2,\ldots,\ell-1\}$. Noting that $\sin\frac{\pi (m+1)j} {\ell}=\pm \sin\frac{\pi r_i} {\ell}$ for every $j$, the claim follows.  As $\gcd(m+1-\ell,m+1)=\gcd(\ell,m+1)=1$, a similar argument indicates that $\prod_{j=1}^{m-\ell}U_m\left(\cos\frac{\pi j}{m+1-\ell}\right)$ also equals $\pm 1$. This completes the proof of Lemma \ref{res1}.
\end{proof}
\begin{corollary}\label{fd1}
	Let $\ell, m,n$ be integers with $2\le \ell \le(m+1)/2$ and $\gcd(\ell,m+1)=1$. Then for any $i\in\{1,2,\ldots,n\}$,
	\begin{equation*}
		\prod_{j=1}^m S_{\ell-1}(\mu_i^{(j)})S_{m-\ell}(\mu_{i}^{(j)})=\pm 1.
	\end{equation*}
\end{corollary}
\begin{proof}
	Note that $S_m(x)-\lambda_i S_{\ell-1}(x)S_{m-\ell}(x)$ is a monic polynomial whose zeroes are $\mu_i^{(1)}$, $\mu_i^{(2)}$, $\ldots$, $\mu_i^{(m)}$. It follows from Lemmas \ref{basicres} and \ref{res1} that
	\begin{eqnarray*}
		\prod_{j=1}^m S_{\ell-1}(\mu_i^{(j)})S_{m-\ell}(\mu_{i}^{(j)})&=&\Res(S_m(x)-\lambda_i S_{\ell-1}(x)S_{m-\ell}(x),S_{\ell-1}(x)S_{m-\ell}(x))\\
		&=&\Res(S_m(x),S_{\ell-1}(x)S_{m-\ell}(x))\\
		&=&\left(\frac{1}{2}\right)^{m(m-1)}\Res(U_m(x),U_{\ell-1}(x)U_{m-\ell}(x))\\
		&=&\pm 1.
	\end{eqnarray*}
	This completes the proof of Corollary \ref{fd1}.
\end{proof}
\begin{lemma}\label{res2}
	 For any integers $\ell, m$ with $2\le \ell \le(m+1)/2$ and $\gcd(\ell,m+1)=1$,
	$$	\Res\left(U_m(x)+tU_{\ell-1}(x)U_{m-\ell}(x),U_{\ell-1}(x)\sum_{k=0}^{m-1}U_k(x)-U_m(x)\sum_{k=0}^{\ell-2}U_k(x)\right)=\pm t^{\lfloor\frac{m}{2}\rfloor}2^{m(m-1)},$$
	where $t$ is any complex number.
\end{lemma}
\begin{proof}
Let $U(x)=U_{\ell-1}(x)\sum_{k=0}^{m-1}U_k(x)-U_m(x)\sum_{k=0}^{\ell-2}U_k(x)$. Then we see from Definition \ref{dS} that $U(x)=S(2x)$. It follows from Remark \ref{rmS} that  $U(x)$ is of degree $m-1$ with leading coefficient $2^{m-1}$. Since $U_m(x)+tU_{\ell-1}(x)U_{m-\ell}(x)$ is of degree $m$, the resultant factors as 
\begin{equation}\label{res3}
	(-1)^{m(m-1)}2^{m(m-1)}\prod_{j=1}^{m-1}(U_m(\beta_j)+tU_{\ell-1}(\beta_j)U_{m-\ell}(\beta_j))
\end{equation}
where $\beta_1,\ldots,\beta_{m-1}$ are zeroes of $U(x)$.

\noindent\textbf{Claim 1}: The $m-1$ zeroes of $U(x)$ are $\cos\frac{2\pi j_1}{m+1}$ ($1\le j_1\le \lfloor\frac{m}{2}\rfloor$), $\cos\frac{2\pi j_2}{\ell}$ ($1\le j_2\le \lfloor\frac{\ell-1}{2}\rfloor$) and $\cos\frac{2\pi j_3}{m+1-\ell}$ ($1\le j_3\le \lfloor\frac{m-\ell}{2}\rfloor$).

Using the trigonometric identity
$$\sin k\theta=\frac{\cos(k-\frac{1}{2})\theta-\cos(k+\frac{1}{2})\theta}{2\sin \frac{\theta}{2}}$$
we conclude that
\begin{eqnarray*}
	U(\cos\theta)&=&U_{\ell-1}(\cos\theta)\sum_{k=0}^{m-1}U_k(\cos\theta)-U_m(\cos\theta)\sum_{k=0}^{\ell-2}U_k(\cos\theta)\\
	&=&\frac{\sin\ell\theta}{\sin^2\theta}\sum_{k=1}^m \sin k\theta-\frac{\sin(m+1)\theta}{\sin^2\theta}\sum_{k=1}^{\ell-1} \sin k\theta\\
	&=&\frac{(\sin\ell\theta)(\cos\frac{\theta}{2}-\cos(m+\frac{1}{2})\theta)-(\sin(m+1)\theta)(\cos\frac{\theta}{2}-\cos(\ell-\frac{1}{2})\theta)}{2\sin^2\theta\sin\frac{\theta}{2}}\\
		&=&\frac{\sin\ell\theta \sin\frac{m+1}{2}\theta\sin\frac{m}{2}\theta-\sin(m+1)\theta\sin\frac{\ell}{2}\theta \sin\frac{\ell-1}{2}\theta}{\sin^2\theta\sin\frac{\theta}{2}}\\
		&=&\frac{2\sin\frac{\ell}{2}\theta \sin\frac{m+1}{2}\theta(\cos\frac{\ell}{2}\theta\sin\frac{m}{2}\theta-\cos\frac{m+1}{2}\theta\sin\frac{\ell-1}{2}\theta)}{\sin^2\theta\sin\frac{\theta}{2}}.
\end{eqnarray*}
It follows that $\cos\frac{2\pi j_1}{m+1}$ ($1\le j_1\le \lfloor\frac{m}{2}\rfloor$) and $\cos\frac{2\pi j_2}{\ell}$ ($1\le j_2\le \lfloor\frac{\ell-1}{2}\rfloor)$ are zeroes of $U(x)$. Let $$g(\theta)=\cos\frac{\ell}{2}\theta\sin\frac{m}{2}\theta-\cos\frac{m+1}{2}\theta\sin\frac{\ell-1}{2}\theta.$$
Then we have 
\begin{eqnarray*}
	&&	g\left(\frac{2\pi j_3}{m+1-\ell}\right)\\&=&\cos\frac{\pi \ell j_3}{m+1-\ell}\sin\frac{\pi m j_3}{m+1-\ell}-\cos\frac{\pi(m+1)j_3}{m+1-\ell}\sin\frac{\pi(\ell-1)j_3}{m+1-\ell}\\
	&=&\cos\frac{\pi \ell j_3}{m+1-\ell}\sin\left(\pi j_3+\frac{\pi (\ell-1) j_3}{m+1-\ell}\right)-\cos\left(\pi j_3+\frac{\pi\ell j_3}{m+1-\ell}\right)\sin\frac{\pi(\ell-1)j_3}{m+1-\ell}\\
	&=&(-1)^{j_3}\cos\frac{\pi \ell j_3}{m+1-\ell}\sin\frac{\pi (\ell-1) j_3}{m+1-\ell}-(-1)^{j_3}\cos\frac{\pi \ell j_3}{m+1-\ell}\sin\frac{\pi (\ell-1) j_3}{m+1-\ell}\\
	&=&0.
\end{eqnarray*}
This verifies that $\cos\frac{2\pi j_3}{m+1-\ell}$ ($1\le j_3\le \lfloor\frac{m-\ell}{2}\rfloor$) are also zeroes of $U(x)$. We are done if we can show that the total number of the established zeroes is exactly $m-1$, i.e., 
\begin{equation*}
	\left\lfloor\frac{m}{2}\right\rfloor+\left\lfloor\frac{\ell-1}{2}\right\rfloor+\left\lfloor\frac{m-\ell}{2}\right\rfloor=m-1.
\end{equation*}
We claim that $\ell-1$ or $m-\ell$ is even. Suppose to the contrary that both $\ell-1$ and $m-\ell$ are odd. Then $\ell$ is even and $m$ is odd. Thus, $\gcd(\ell,m+1)\ge 2$, contradicting  the assumption of this lemma.  Since $\ell-1$ or $m-\ell$ is even, we have 
\begin{equation*}
	\left\lfloor\frac{\ell-1}{2}\right\rfloor+\left\lfloor\frac{m-\ell}{2}\right\rfloor=\left\lfloor\frac{m-1}{2}\right\rfloor
\end{equation*}
and hence 
\begin{equation*}
	\left\lfloor\frac{m}{2}\right\rfloor+\left\lfloor\frac{\ell-1}{2}\right\rfloor+\left\lfloor\frac{m-\ell}{2}\right\rfloor=	\left\lfloor\frac{m}{2}\right\rfloor+\left\lfloor\frac{m-1}{2}\right\rfloor=m-1.
\end{equation*}
This completes the proof of Claim 1.

Note that $\cos\frac{2\pi j_1}{m+1}$ ($1\le j_1\le \lfloor\frac{m}{2}\rfloor$), $\cos\frac{2\pi j_2}{\ell}$ ($1\le j_2\le \lfloor\frac{\ell-1}{2}\rfloor$) and $\cos\frac{2\pi j_3}{m+1-\ell}$ ($1\le j_3\le \lfloor\frac{m-\ell}{2}\rfloor$) are zeroes of $U_m(x)$, $U_{\ell-1}(x)$ and $U_{m-\ell}(x)$, respectively. Then the product $\prod_{j=1}^{m-1}(U_m(\beta_j)+tU_{\ell-1}(\beta_j)U_{m-\ell}(\beta_j)$ in Eq.~\eqref{res3} simplifies to
\begin{equation}\label{ressim}
\prod_{j_1=1}^{\lfloor\frac{m}{2}\rfloor}\left(tU_{\ell-1}\left(\cos\frac{2\pi j_1}{m+1}\right)U_{m-\ell}\left(\cos\frac{2\pi j_1}{m+1}\right)\right)\prod_{j_2=1}^{\lfloor\frac{\ell-1}{2}\rfloor}U_{m}\left(\cos\frac{2\pi j_2}{\ell}\right)\prod_{j_3=1}^{\lfloor\frac{m-\ell}{2}\rfloor}U_{m}\left(\cos\frac{2\pi j_3}{m+1-\ell}\right).
\end{equation}

\noindent\textbf{Claim 2}: It holds that \begin{equation}\label{ex1}\prod_{j_1=1}^{\lfloor\frac{m}{2}\rfloor}U_{\ell-1}\left(\cos\frac{2\pi j_1}{m+1}\right)=\pm 1.\end{equation}

Note that 
\begin{equation}\label{um}\prod_{j_1=1}^{\lfloor\frac{m}{2}\rfloor}U_{\ell-1}\left(\cos\frac{2\pi j_1}{m+1}\right)=\prod_{j_1=1}^{\lfloor\frac{m}{2}\rfloor}\frac{\sin\frac{2\pi \ell j_1}{m+1}}{\sin\frac{2\pi j_1}{m+1}}.\end{equation}
We prove Eq.~\eqref{ex1} by considering the parity of $m$. First assume that $m$ is odd. Let $m'=(m+1)/2$. Then $\lfloor\frac{m}{2}\rfloor=\frac{m-1}{2}=m'-1$ and hence we can rewrite Eq.~\eqref{um} as
\begin{equation}\label{um2}
\prod_{j_1=1}^{\lfloor\frac{m}{2}\rfloor}U_{\ell-1}\left(\cos\frac{2\pi j_1}{m+1}\right)=\prod_{j_1=1}^{m'-1}\frac{\sin\frac{\pi \ell j_1}{m'}}{\sin\frac{\pi j_1}{m'}}.
\end{equation}
Since $\gcd(\ell,m+1)=1$, we must have  $\gcd(\ell,m')=1$. Thus $\{\ell j_1\colon\, 1\le j_1\le m'-1\}$ equals $\{1,2,\ldots,m'-1\}$ module $m'$. This means that the right-hand side of Eq.~\eqref{um2} equals $\pm 1$, completing the  proof of Eq.~\eqref{ex1} for the case that $m$ is odd.

Now assume that $m$ is even. We have 
\begin{equation}\label{m2}
\left(	\prod_{j_1=1}^{\lfloor\frac{m}{2}\rfloor}\frac{\sin\frac{2\pi \ell j_1}{m+1}}{\sin\frac{2\pi j_1}{m+1}}\right)^2=\prod_{j_1=1}^{\frac{m}{2}}\frac{\sin\frac{2\pi \ell j_1}{m+1}\sin\frac{2\pi \ell j_1}{m+1}}{\sin\frac{2\pi j_1}{m+1}\sin\frac{2\pi j_1}{m+1}}=\prod_{j_1=1}^{\frac{m}{2}}\frac{\sin\frac{2\pi \ell j_1}{m+1}\sin\frac{2\pi \ell (-j_1)}{m+1}}{\sin\frac{2\pi j_1}{m+1}\sin\frac{2\pi (-j_1)}{m+1}}=\prod_{\substack{j_1=-\frac{m}{2}\\j_1\neq 0}}^{\frac{m}{2}}\frac{\sin\frac{2\pi \ell j_1}{m+1}}{\sin\frac{2\pi j_1}{m+1}}.
\end{equation}
Let $J=\{-\frac{m}{2},-\frac{m}{2}+1,\ldots,-1\}\cup \{1,2,\ldots,\frac{m}{2}\}$. Note that $J$, together with $0$, constitutes a complete residue system modulo $m+1$.   As $\gcd(2,m+1)=1$, we conclude that $\{2j_1\colon\,j_1\in J\}$ equals $J$, modulo $m+1$. Similarly, as $\gcd(\ell,m+1)=1$ and $\gcd(2,m+1)=1$, we find that $\gcd(2\ell,m+1)=1$ and hence $\{2\ell j_1\colon\,j_1\in J\}$ also equals $J$, modulo $m+1$. Therefore,
\begin{equation}\label{pm}
	\prod_{\substack{j_1=-\frac{m}{2}\\j_1\neq 0}}^{\frac{m}{2}}\frac{\sin\frac{2\pi \ell j_1}{m+1}}{\sin\frac{2\pi j_1}{m+1}}=	\frac{\prod_{j_1\in J}\sin\frac{2\pi \ell j_1}{m+1}}{\prod_{j_1\in J}\sin\frac{2\pi j_1}{m+1}}=\pm 	\frac{\prod_{j_1\in J}\sin\frac{\pi j_1}{m+1}}{\prod_{j_1\in J}\sin\frac{\pi j_1}{m+1}}=\pm 1.
\end{equation}
It follows from Eq.~\eqref{m2} that the sign `$\pm$' in Eq.~\eqref{pm} is indeed a `$+$' and 
\begin{equation*}
	\prod_{j_1=1}^{\lfloor\frac{m}{2}\rfloor}\frac{\sin\frac{2\pi \ell j_1}{m+1}}{\sin\frac{2\pi j_1}{m+1}}=\pm 1.
\end{equation*}
This, combining with Eq.~\eqref{um}, completes the proof of Claim 2.

Using a similar argument as the proof of Claim 2, we can show that all the products $ \prod_{j_1=1}^{\lfloor\frac{m}{2}\rfloor}U_{m-\ell}\left(\cos\frac{2\pi j_1}{m+1}\right)$, $\prod_{j_2=1}^{\lfloor\frac{\ell-1}{2}\rfloor}U_{m}\left(\cos\frac{2\pi j_2}{\ell}\right)$ and $\prod_{j_3=1}^{\lfloor\frac{m-\ell}{2}\rfloor}U_{m}\left(\cos\frac{2\pi j_3}{m+1-\ell}\right)$ equal $\pm 1$. Thus, Eq.~\eqref{ressim} reduces to $\pm t^{\lfloor\frac{m}{2}\rfloor}$ and the lemma follows by Eq.~\eqref{res3}.		
\end{proof}
\begin{corollary}\label{fd2}	Let $\ell, m,n$ be integers with $2\le \ell \le(m+1)/2$ and $\gcd(\ell,m+1)=1$. Then for any $i\in\{1,2,\ldots,n\}$,
	\begin{equation*}
		\prod_{j=1}^m S(\mu_i^{(j)})=\pm \lambda_i^{\lfloor\frac{m}{2}\rfloor}.
	\end{equation*} 
\end{corollary}
\begin{proof}
Recall that 	$S(x)=S_{\ell-1}(x)\sum_{k=0}^{m-1}S_k(x)-S_m(x)\sum_{k=0}^{\ell-2}S_k(x)$ is a monic polynomial of degree $m-1$, and $\mu_i^{(j)}$'s are roots of $S_m(x)-\lambda_iS_{\ell-1}(x)S_{m-\ell}(x)$, which is a monic polynomial of degree $m$. It follows from Lemmas  \ref{basicres} and  \ref{res2} that
\begin{eqnarray*}
&&\prod_{j=1}^m S(\mu_i^{(j)})\\
&=&\Res\left(S_m(x)-\lambda_iS_{\ell-1}(x)S_{m-\ell}(x),S_{\ell-1}(x)\sum_{k=0}^{m-1}S_k(x)-S_m(x)\sum_{k=0}^{\ell-2}S_k(x)\right)\\
&=&\left(\frac{1}{2}\right)^{m(m-1)}\Res\left(U_m(x)-\lambda_iU_{\ell-1}(x)U_{m-\ell}(x),U_{\ell-1}(x)\sum_{k=0}^{m-1}U_k(x)-U_m(x)\sum_{k=0}^{\ell-2}U_k(x)\right)\\
&=&\pm\lambda_i^{\lfloor\frac{m}{2}\rfloor}.
\end{eqnarray*}
This completes  the proof of Corollary \ref{fd2}.
\end{proof}
\noindent\textbf{Proof of Theorem \ref{main}} Let $\Delta_1$ and $\Delta_2$ be the two numbers defined in Eqs.~\eqref{d1} and \eqref{d2}. By Corollaries  \ref{fd1} and \ref{fd2}, we conclude that $\Delta_1=\pm 1$ and $\Delta_2=\pm (\prod_{i=1}^n\lambda_i)^{\lfloor\frac{m}{2}\rfloor}$. As $\lambda_i$'s are eigenvalues of $A(G)$, we have $\prod_{i=1}\lambda_i=\det A(G)$. It follows from Proposition \ref{dwt} that
	\begin{equation*}
	\det W(G\circ P_m^{(\ell)})=\pm \Delta_1\Delta_2\left(\det W(G)\right)^m=\pm (\det A(G))^{\lfloor\frac{m}{2}\rfloor}\left(\det W(G)\right)^m,
\end{equation*}
completing the proof of Theorem \ref{main}.

\section{Proof of Theorem \ref{cons2}}
Recall that \begin{equation*}
	\mathcal{F}=\bigcup_{n\text{~even}} \mathcal{F}_n
\end{equation*}
where $\mathcal{F}_n$ is the collection of all $n$-vertex graphs $G$ such that $\det A(G)=\pm 1$ and $\det W(G)=\pm 2^{\lfloor n/2\rfloor}$.
The key result of this section is the following proposition on the determinant of $A(G\circ P_m^{(\ell)})$.
\begin{proposition}\label{dc}
	 Let $G$ be any graph with $\det A(G)=1$. Then $\det A(G\circ P_m^{(\ell)})=\pm 1$ for any two integers  $m$ and $\ell$ with $1\le \ell\le m$ and $\gcd(\ell,m+1)=1$.
\end{proposition}
\begin{proof}
Noting that $\phi(G\circ P_m^{(\ell)},0)$, the constant term of the characteristic polynomial $\phi(G\circ P_m^{(\ell)};x)$ is $\pm \det A(G\circ P_m^{(\ell)})$,  Corollary \ref{chagpk} implies that
\begin{equation}\label{dd}
	 \det A(G\circ P_m^{(\ell)})=\pm \prod_{i=1}^{n}(S_m(0)-\lambda_i S_{\ell-1}(0)S_{m-\ell}(0)), 
\end{equation}
where $\lambda_1,\ldots,\lambda_n$ are the eigenvalues of $G$.

Recalling that $S_0(x)=1, S_1(x)=x$ and $S_k(x)=xS_{k-1}(x)-S_{k-2}(x)$ and letting $x=0$, we find that 
\begin{equation*}
	S_k(0)=\begin{cases}
		\pm 1,&\text{ $k$ even,}\\
		0,&\text{ $k$ odd.}
	\end{cases}
\end{equation*}
We first consider the case that $m$ is odd. As $\gcd(\ell,m+1)=1$, we see that $\ell$ must be odd. This means that both $\ell-1$ and $m-\ell$ are even. Thus, $S_m(0)=0, S_{\ell-1}(0)S_{m-\ell}(0)=\pm 1$   and hence  Eq.~\eqref{dd} reduces to 
\begin{equation*}
	\det(A(G\circ P_m^{(\ell)}))=\pm \prod_{i=1}^{n}\lambda_i =\pm \det A(G)=\pm 1.
\end{equation*}
Now assume that $m$ is even. Note that the sum of $\ell-1$ and $m-\ell$ is $m-1$, which is odd. Then, $\ell-1$ or $m-\ell$ is odd, which implies that $S_{\ell-1}(0)S_{m-\ell}(0)=0$. Now Eq.~\eqref{dd} reduces to 
\begin{equation*}
	\det(A(G\circ P_m^{(\ell)}))=\pm \prod_{i=1}^{n}S_m(0)=\pm 1, 
\end{equation*}
This completes the proof of  Proposition \ref{dc}.
\end{proof}
\begin{proposition}\label{cons2sim}
Let $G\in \mathcal{F}$. Then $G\circ P_m^{(\ell)}\in \mathcal{F}$ for any two integers  $m$ and $\ell$ with $1\le \ell\le m$ and $\gcd(\ell,m+1)=1$.
\end{proposition}
\begin{proof}
	Let $n$ be the order of $G$. Then $n$ is even, $\det A(G)=\pm 1$ and $\det W(G)=2^\frac{n}{2}$. By Theorem \ref{main}, we have 
	$$\det W(G\circ P_m^{(\ell)})=\pm (\det A(G))^{\lfloor\frac{m}{2}\rfloor}(\det W(G))^m=\pm 2^\frac{mn}{2}.$$
	Furthermore, we have $\det W(G\circ P_m^{(\ell)})=\pm 1$ by Proposition \ref{dc}. Therefore, $G\circ P_m^{(\ell)}\in \mathcal{F}$, as desired.
\end{proof}
Since any graph in $\mathcal{F}$ is DGS \cite{wang2017}, Theorem \ref{cons2} clearly follows by Proposition \ref{cons2sim}.

\section*{Declaration of competing interest}
There is no conflict of interest.
\section*{Acknowledgments}
This work is supported by the National Natural Science Foundation of China (Grant No. 12001006).


\begin{thebibliography}{99}
\bibitem{choi2021}J.~Choi, S.~Moon, S.~Park, A note on the invariant factors of the walk matrix of a graph Linear Algebra Appl. 631(15)(2021) 362--378.
		
\bibitem{fulkerson}D.~R.~Fulkerson,  O.~A.~Gross, Incidence matrices and interval graphs, Pacific J. Math. 15(3)(1965) 835--855.
\bibitem{gutman1980} I. Gutman, Spectral properties of some graphs derived from bipartite graphs. MATCH Commun. Math. Comput. Chem. 8(1980) 291--314.
	
\bibitem{godsil2012} C.~D.~Godsil, Controllable subsets in graphs,  Ann. Combin. 16(4)(2012) 733--744. 
	
\bibitem{godsil1978} C.~D.~Godsil, B.~D.~McKay, A new graph product and its spectrum. Bull. Austral. Math. Soc. 18(1978) 21--28.
	
\bibitem{gm1981} C.~D.~Godsil,  B.~D.~McKay,  Spectral conditions for the reconstructibility of a graph, J. Combin. Theory Ser. B 30(3)(1981) 285--289.
		
\bibitem{liu2022}	F.~Liu, J.~Siemons,  Unlocking the walk matrix of a graph, J Algebra Combin. 55(2022) 663--690.

\bibitem{moon2023} S.~Moon, S.~Park, The Smith normal form of the walk matrix of the extended Dynkin graph $\tilde{D}_n$, Linear Algebra Appl. 678(1)(2023) 169--190.

\bibitem{rourke2016} O'Rourke, Sean, T. Behrouz, On a conjecture of Godsil concerning controllable random graphs. SIAM J. Control Optim. 54(6)(2016) 3347--3378.

\bibitem{wang2017}	W. Wang, A simple arithmetic criterion for graphs being determined by their generalized spectra, 	J. Combin. Theory, Ser. B 122(2017) 438--451.

\bibitem{wang2013} W. Wang, Generalized spectral characterization revisited, Electron. J. Comb. 20(4)(2013) \#P4.

\bibitem{wang2021} W.~Wang, On the Smith normal form of walk matrices, Linear Algebra Appl. 612(2021) 30--41.

\bibitem{wwz2023} W.~Wang, W.~Wang, F.~Zhu, An improved condition for a graph to be determined by its generalized spectrum, European J. Combin. 108(2023) 103638.

\bibitem{ww2024} W.~Wang, W.~Wang, Haemers' conjecture: An algorithmic perspective, Exp. Math. 1--28. https://doi.org/10.1080/10586458.2024.2337229.

\bibitem{mao2015}L.~Mao, F.~Liu, W.~Wang, A new method for constructing graphs determined by their generalized spectrum, Linear Algebra Appl. 477(2015) 112--127. 

\bibitem{mao2022} L.~Mao, W.~Wang, Generalized spectral characterization of rooted product graphs, Linear Multilinear Algebra  71(14)(2023) 2310--2324.

\bibitem{shan}H.~Shan, X.~Liu,  Exploring graphs with distinct $M$-eigenvalues: product operation,
Wronskian vertices, and controllability, arXiv2412.18759.

\bibitem{schwenk1974}	A.J. Schwenk, Computing the characteristic polynomial of a graph, Graphs Combin. (R. Bari and F. Harary, eds.), Springer, Berlin, (1974), 153--172.

\bibitem{wym2024} W.~Wang, Z.~Yan, L.~Mao, Proof of a conjecture on the determinant of the walk matrix of rooted product with a path, Linear and Multilinear Algebra  72(5)(2024) 828--840.

 \end{thebibliography}
\end{document}